\documentclass[12pt]{amsart}

\usepackage[margin=1in]{geometry}
\usepackage{amsmath,amsthm,amssymb}
\usepackage{dsfont}
\usepackage{amsmath,amscd}
\usepackage{mathrsfs}
\usepackage{color}
\usepackage{hyperref}
\usepackage{enumitem}
\usepackage{array}

\hypersetup{
    colorlinks=true, 
    linkcolor=blue, 
    urlcolor=red, 
    linktoc=all 
}

\input xy
\xyoption{all}

\newenvironment{remark}{\noindent\textbf{Remark:}\hspace{1em}}{}

\newcommand{\Z}{\mathbb{Z}}

\newcommand{\one}{\mathds{1}}

\newcommand{\ad}[2]{Ad_{#2}{#1}}
\newcommand{\sm}[1]{\langle{#1}\rangle}
\newcommand{\cm}[1]{c({#1})}

\theoremstyle{plain}
\newtheorem{thm}{Theorem}[section]
\newtheorem{prop}[thm]{Proposition}
\newtheorem{lemma}[thm]{Lemma}
\newtheorem{cor}[thm]{Corollary}

\theoremstyle{definition}
\newtheorem{dfn}{Definition}[section]

\begin{document}

\title{The odd primary order of the commutator on low rank Lie groups}
\author{Tseleung So}
\address{Mathematical Sciences, University of Southampton, SO17 1BJ, UK}
\email{tls1g14@soton.ac.uk}
\thanks{}



\date{}

\begin{abstract}
Let $G$ be a simple, simply-connected, compact Lie group of low rank relative to a fixed prime~$p$. After localization at $p$, there is a space $A$ which ``generates'' $G$ in a certain sense. Assuming $G$ satisfies a homotopy nilpotency condition relative to $p$, we show that the Samelson product $\sm{\one_G,\one_G}$ of the identity of $G$ equals the order of the Samelson product~$\sm{\imath,\imath}$ of the inclusion $\imath:A\to G$. Applying this result, we calculate the orders of~$\sm{\one_G,\one_G}$ for all~$p$-regular Lie groups and give bounds of the orders of $\sm{\one_G,\one_G}$ for certain quasi-$p$-regular Lie groups.
\end{abstract}

\maketitle

\section{Introduction}
In this paper, $G$ is a simple, simply-connected, compact Lie group and $p$ is an odd prime. By a theorem of Hopf, $G$ is rationally homotopy equivalent to a product of spheres~$\prod^l_{i=1}S^{2n_i-1}$, where $n_1\leq\cdots\leq n_l$. The sequence $(2n_1-1,\cdots,2n_l-1)$ is called the \emph{type} of $G$. Localized at $p$, it is known \cite{CN84,MNT77} that $G$ is homotopy equivalent to a product of H-spaces $\prod^{p-1}_{i=1}B_i$, and there exists a co-H-space $A$ and a map $\imath:A\to G$ such that $H_*(G)$ is the exterior algebra generated by $\imath_*(\tilde{H}_*(A))$. For $1\leq i\leq l$, if $B_i$ is $S^{2n_i-1}$, then we call $G$ \emph{$p$-regular}. If each $B_i$ is either $S^{2n_i-1}$ or $B(2n_i-1,2n_i+2p-3)$ that is the $S^{2n_i-1}$-bundle over $S^{2n_i+2p-3}$ classified by $\frac{1}{2}\alpha\in\pi_{2n_i+2p-4}(S^{2n_i-1})$, then we call $G$ \emph{quasi-$p$-regular}.

For any maps $f:X\to G$ and $g:Y\to G$, let $c(f,g):X\times Y\to G$ be a map sending~\mbox{$(x,y)\in X\times Y$} to their commutator $[x, y]=f(x)^{-1}g(y)^{-1}f(x)g(y)$. Then $c(f,g)$ descends to a map $\sm{f,g}:X\wedge Y\to G$. The map $\sm{f,g}$ is called the \emph{Samelson product} of $f$ and~$g$. The \emph{order} of $\sm{f,g}$ is defined to be the minimum number $k$ such that the composition
\[
k\circ\sm{f,g}:X\wedge Y\overset{\sm{f,g}}{\longrightarrow}G\overset{k}{\longrightarrow}G
\]
is null-homotopic, where $k:G\to G$ is the $k^\text{th}$-power map. In particular, when $f$ and $g$ are the identity map $\one_G$ of $G$, the Samelson product $\sm{\one_G,\one_G}$ is universal and we are interested in finding its order.

There is a notion of nilpotency in homotopy theory analogous to that for groups. Let~$c_1$ be the commutator map $c(\one_G,\one_G):G\times G\to G$, and let~\mbox{$c_n=c_1\circ(c_{n-1}\times\one_G)$} be the $n$-iterated commutator for $n>1$. The \emph{homotopy nilpotence class} of $G$ is the number $n$ such that $c_n$ is null-homotopic but $c_{n-1}$ is not. In certain cases the homotopy nilpotence class of $p$-localized $G$ is known. Kaji and Kishimoto \cite{kk10} showed that $p$-regular Lie groups have homotopy nilpotence class at most 3. When $G$ is quasi-$p$-regular and $p\geq7$, Kishimoto \cite{kishimoto09} showed that $SU(n)$ has homotopy nilpotence class at most 3, and Theriault \cite{theriault16} showed that exceptional Lie groups have homotopy nilpotence class at most 2.

Here we restrict $G$ to be a Lie group having low rank with respective to an odd prime $p$. That is, $G$ and $p$ satisfy:
\begin{equation}\label{retractile_list}
\begin{array}{l l}
SU(n)		&n\leq(p-1)(p-2)+1\\
Sp(n)		&2n\leq(p-1)(p-2)\\
Spin(2n+1)	&2n\leq(p-1)(p-2)\\
Spin(2n)	&2n-2\leq(p-1)(p-2)\\
G_2,F_4,E_6	&p\geq5\\
E_7,E_8		&p\geq7,
\end{array}
\end{equation}
In these cases, Theriault \cite{theriault07} showed that $\Sigma A$ is a retract of $\Sigma G$. Let $\sm{\imath,\imath}$ be the composition
\[
\sm{\imath,\imath}:A\wedge A\overset{\imath\wedge\imath}{\hookrightarrow}G\wedge G\overset{\sm{\one_G,\one_G}}{\longrightarrow}G.
\]
Then obviously the order of $\sm{\one_G,\one_G}$ is always greater than or equal to the order of $\sm{\imath,\imath}$. Conversely, we show that the order of $\sm{\imath,\imath}$ restricts the order of $\sm{\one_G,\one_G}$ under certain conditions.

\begin{thm}\label{main1}
Let $G$ be a compact, simply-connected, simple Lie group of low rank and let~$p$ be an odd prime. Localized at $p$, if the homotopy nilpotence class of $G$ is less than~$p^r+1$, then the order of the Samelson product $\sm{\one_G,\one_G}$ is $p^r$ if and only if the order of~$\sm{\imath,\imath}$ is $p^r$.
\end{thm}

The strategy for proving Theorem~\ref{main1} is to extend $A\to G$ to an H-map $\Omega\Sigma A\to G$ which has a right homotopy inverse, that is to retract $[G\wedge G, G]$ off $[\Omega\Sigma A\wedge\Omega\Sigma A, G]$, and use commutator calculus to analyze the latter. Combine Theorem~\ref{main1} and the known results in~\cite{kk10, kishimoto09, theriault16} to get the following statement.

\begin{cor}\label{cor_intro}
The order of $\sm{\one_G,\one_G}$ equals the order of $\sm{\imath,\imath}$ when
\begin{itemize}
\item	$G$ is $p$-regular or;
\item	$p\geq7$ and $G$ is a quasi-$p$-regular Lie group which is one of $SU(n),F_4,E_6,E_7$ or $E_8$.
\end{itemize}
\end{cor}

On the one hand, there is no good method to calculate the order of $\sm{\one_G, \one_G}$ in general. A direct computation is not practical since one has to consider all the cells in $G\wedge G$ and their number grows rapidly when there is a slight increase in the rank of $G$. On the other hand, Corollary~\ref{cor_intro} says that for all $p$-regular Lie groups and most of the quasi-$p$-regular Lie groups, we can determine the order of $\sm{\one_G,\one_G}$ by computing the order of $\sm{\imath,\imath}$. The latter is easier to work with since $A$ has a much simpler CW-structure than $G$. To demonstrate the power of Theorem~\ref{main1}, we apply this result to compute the order of $\sm{\one_G,\one_G}$ for all $p$-regular cases and some quasi-$p$-regular cases.

\begin{thm}
For a $p$-localized Lie group $G$, the order of $\sm{\one_G,\one_G}$ is $p$ when
\[
\begin{array}{|c|c|c|}
\hline
G		&r=0	&r=1\\
\hline
SU(n)		&p>2n	&n\leq p<2n\\
Sp(n)		&p>4n	&2n<p<4n\\
Spin(2n+1)	&p>4n	&2n<p<4n\\
Spin(2n)	&p>4n-4	&2n-2<p<4n-4\\
G_2			&p=5,p>11	&p=7,11\\
F_4,E_6		&p>23		&11\leq p\leq23\\
E_7			&p>31		&17\leq p\leq31\\
E_8			&p>59		&23\leq p\leq59\\
\hline
\end{array}
\]
\end{thm}
\noindent
For many other quasi-$p$-regular cases, we give rough bounds on the order of $\sm{\one_G,\one_G}$ by bounding the order of $\sm{\imath,\imath}$.

Here is the structure of this paper. In Section 2 we prove Theorem~\ref{main1} assuming Lemma~\ref{lemma_sm_mk_null}, whose proof is given in Section 3 because of its length. Section 3 is divided into two parts. In the first part we consider the algebraic properties of Samelson products and in the second part we use algebraic methods to prove Lemma~\ref{lemma_sm_mk_null}. In Section 4 we apply Theorem~\ref{main1} and use other known results to calculate bounds on the order of the Samelson product $\sm{\one_G,\one_G}$ for quasi-$p$-regular Lie groups.

\section{Samelson products of low rank Lie groups}
\begin{dfn}
Let $G$ be a simple, simply-connected, compact Lie group and $p$ be an odd prime. Localized at $p$, a triple $(A, \imath, G)$ is \emph{retractile} if $A$ is a co-H-space and a subspace of $G$ and $\imath:A\hookrightarrow G$ is an inclusion such that
\begin{itemize}
\item	there is an algebra isomorphism $H_*(G)\cong\Lambda(\tilde{H}_*(A))$ of homologies with mod-$p$ coefficients;
\item	the induced homomorphism $\imath_*: H_*(A)\to H_*(G)$ is an inclusion of the generating set;
\item	the suspension $\Sigma\imath:\Sigma A\to\Sigma G$ has a left homotopy inverse $t:\Sigma G\to\Sigma A$.
\end{itemize}
We also refer to $G$ as being retractile for short.
\end{dfn}

From now on, we take $p$-localization and assume $G$ and $p$ satisfy~(\ref{retractile_list}). According to \cite{theriault07},~$G$ is retractile. First we want to establish a connection between $G$ and $\Omega\Sigma A$. Consider the homotopy commutative diagram
\begin{equation}\label{CD1}
\xymatrix{
A\ar[dr]_{\Sigma}\ar[rr]^{\imath}	&&G\\
&\Omega\Sigma A\ar[ur]_{\tilde{\imath}}	&
}
\end{equation}
where $\Sigma:A\to\Omega\Sigma A$ is the suspension and $\tilde{\imath}:\Omega\Sigma A\to G$ is an H-map. Since $G$ is retractile, the suspension $\Sigma\imath:\Sigma A\to\Sigma G$ has a left homotopy inverse $t:\Sigma G\to\Sigma A$. Let $s$ be the composition
\[
s:G\overset{\Sigma}{\longrightarrow}\Omega\Sigma G\overset{\Omega t}{\longrightarrow}\Omega\Sigma A.
\]

\begin{lemma}\label{lemma_rho+s}
The map $\tilde{\imath}\circ s$ is a homotopy equivalence.
\end{lemma}

\begin{proof}
Denote the composition $\tilde{\imath}\circ s$ by $e$ for convenience. Consider the commutative diagram
\[\xymatrix{
A\ar[d]_{\Sigma}\ar[r]^{\imath}				&G\ar[d]_{\Sigma}\ar[dr]^{s}		&\\
\Omega\Sigma A\ar[r]^{\Omega\Sigma\imath}	&\Omega\Sigma G\ar[r]^{\Omega t}	&\Omega\Sigma A.
}\]
The commutativity of the left square is due to the naturality of the suspension map, and the commutativity of the right triangle follows from the definition of $s$. The bottom row is homotopic to the identity since $t$ is a left homotopy inverse for $\Sigma\imath$. Hence we have $s\circ\imath\simeq\Sigma$ and consequently
\[
e\circ\imath=\tilde{\imath}\circ s\circ\imath\simeq\tilde{\imath}\circ\Sigma.
\]
By Diagram~(\ref{CD1}) $\tilde{\imath}\circ\Sigma$ is homotopic to $\imath$. This implies that $(e\circ\imath)_*$ sends $H_*(A)$ onto the generating set of $H_*(G)=\Lambda(\tilde{H}_*(A))$ where we consider the mod-$p$ homology. Dually,~\mbox{$(e\circ\imath)^*:H^*(G)\to H^*(A)$} is an epimorphism. The generating set $\imath^*(H^*(A))$ is in $Im(e^*)$. Since $e^*:H^*(G)\to H^*(G)$ is an algebra map, $e^*$ is an epimorphism and hence is an isomorphism. Therefore \mbox{$e:G\to G$} is a homotopy equivalence.
\end{proof}

We claim that the Samelson product
\[
\sm{\tilde{\imath}, \tilde{\imath}}:\Omega\Sigma A\wedge\Omega\Sigma A\overset{\tilde{\imath}\wedge\tilde{\imath}}{\longrightarrow}G\wedge G\overset{\sm{\one_G, \one_G}}{\longrightarrow}G
\]
has the same order as $\sm{\one_G,\one_G}$.

\begin{lemma}\label{rho}
The map $p^r\circ\sm{\one_G, \one_G}$ is null-homotopic if and only if $p^r\circ\sm{\tilde{\imath}, \tilde{\imath}}$ is null-homotopic.
\end{lemma}

\begin{proof}
The sufficiency part is obvious. We only show the necessity part. Suppose $p^r\circ\sm{\tilde{\imath}, \tilde{\imath}}$ is null-homotopic. By Lemma~\ref{lemma_rho+s}, $e=\tilde{\imath}\circ s$ is a homotopy equivalence. Composing with its inverse $e'$, the map $\tilde{\imath}\circ s\circ e'$ is homotopic to the identity. Then we obtain
\[
p^r\circ\sm{\one_G, \one_G}\simeq p^r\circ\sm{\tilde{\imath}\circ s\circ e', \tilde{\imath}\circ s\circ e'}=p^r\circ\sm{\tilde{\imath},\tilde{\imath}}\circ(s\circ e'\wedge s\circ e')
\]
which is null-homotopic since $p^r\circ\sm{\tilde{\imath},\tilde{\imath}}$ is null-homotopic.
\end{proof}

Combining Diagram~\ref{CD1} and the fact that $\tilde{\imath}$ is an H-map, we have the commutative diagram
\begin{equation}\label{dgm_rho_e_m}
\xymatrix{
									&(\Omega\Sigma A)^k\ar[d]^{\tilde{\imath}^k}\ar[r]^{\mu^k}	&\Omega\Sigma A\ar[d]^{\tilde{\imath}}\\
A^k\ar[ur]^{j^k}\ar[r]^{\imath^k}	&G^k\ar[r]^{m^k}											&G
}
\end{equation}
where $\mu^k$ and $m^k$ are the $k$-fold multiplications in $\Omega\Sigma A$ and $G$. Let $m_k$ and $e_k$ be the compositions
\[
\begin{array}{c c c}
m_k:A^k\overset{\imath^k}{\longrightarrow}G^k\overset{m^k}{\longrightarrow}G
&\text{and}
&e_k:A^k\overset{j^k}{\longrightarrow}(\Omega\Sigma A)^k\overset{\mu^k}{\longrightarrow}\Omega\Sigma A.
\end{array}
\]
Then we have the following commutative diagram
\[\xymatrix{
A^k\wedge A^l\ar[d]_{e_k\wedge e_l}\ar[dr]^{m_k\wedge m_l}						&											&				&\\
\Omega\Sigma A\wedge\Omega\Sigma A\ar[r]^-{\tilde{\imath}\wedge\tilde{\imath}}	&G\wedge G\ar[r]^-{\sm{\one_G,\one_G}}		&G\ar[r]^{p^r}	&G
}\]
Observe that there is a string of equalities
\begin{eqnarray*}
[\Omega\Sigma A\wedge\Omega\Sigma A, G]
&=&[\Sigma\Omega\Sigma  A\wedge\Omega\Sigma A, BG]\\
&=&[\bigvee^\infty_{k,l=1}\Sigma A^{\wedge k}\wedge A^{\wedge l}, BG]\\
&=&\prod^\infty_{k,l=1}[A^{\wedge k}\wedge A^{\wedge l}, G].
\end{eqnarray*}
The first and the third lines are due to adjunction, and the second line is due to James splitting $\Sigma\Omega\Sigma A\simeq\Sigma^{\infty}_{k=1}\Sigma A^{\wedge k}$. It is not hard to see that the nullity of $p^r\circ\sm{\tilde{\imath}, \tilde{\imath}}$ implies the nullity of the components $p^r\circ\sm{m_k,m_l}$. In the following we show that the converse is true.

\begin{lemma}\label{k component of JA}
Let $X$ be a space and let $f:X\to G$ be a map. If $p^r\circ\sm{m_k,f}:A^k\wedge X\to G$ is null-homotopic for all $k$, then $p^r\circ\sm{\tilde{\imath},f}:\Omega\Sigma A\wedge X\to G$ is null-homotopic. Similarly, if $p^r\circ\sm{f,m_l}:X\wedge A^l\to G$ is null-homotopic for all $l$, then $p^r\circ\sm{f,\tilde{\imath}}:X\wedge\Omega\Sigma A\to G$ is null-homotopic.
\end{lemma}

\begin{proof}
We only prove the first statement since the second statement can be proved similarly. Let $h:\Sigma\Omega\Sigma A\wedge X\to BG$ be the adjoint of $p^r\circ\sm{\tilde{\imath},f}$. It suffices to show that $h$ is null-homotopic.

For any $k$, choose a right homotopy inverse $\psi_k$ of the suspended quotient map~\mbox{$\Sigma A^k\to\Sigma A^{\wedge k}$}, and let $\Psi_k$ be the composition
\[
\Psi_k:\Sigma A^{\wedge k}\overset{\psi_k}{\longrightarrow}\Sigma A^k\overset{\Sigma e_k}{\longrightarrow}\Sigma\Omega\Sigma A.
\]
Observe that $e_k$ is the product of $k$ copies of the suspension $j$ and $j_*:H_*(A)\to H_*(\Omega\Sigma A)$ is the inclusion of the generating set into $H_*(\Omega\Sigma A)\cong T(\tilde{H}_*(A))$. The map~\mbox{$(e_k)_*:H_*(A^k)\to H^*(\Omega\Sigma A)$} sends the submodule $S_k\subset H_*(A^k)\cong H_*(A)^{\otimes k}$ consisting of length $k$ tensor products onto the submodule $M_k\subset T(\tilde{H}_*(A))$ consisting of length $k$ tensor products. Therefore $(\Psi_k)_*$ does the same. Then their sum
\[
\Psi=\bigvee_{k=1}^{\infty}\Psi_k:\bigvee_{k=1}^{\infty}\Sigma A^{\wedge k}\to\Sigma\Omega\Sigma A
\]
induces a homology isomorphism and hence is a homotopy equivalence.

We claim that $h\circ(\Psi_k\wedge\one_X)$ is null-homotopic for all $k$, where $\one_X$ is the identity of $X$. Observe that the adjoint of the composition
\[
\Sigma A^k\wedge X\overset{\Sigma e_k\wedge\one_X}{\longrightarrow}\Sigma\Omega\Sigma A\wedge X\overset{h}{\longrightarrow}BG
\]
is $p^r\circ\sm{\tilde{\imath},f}\circ(e_k\wedge\one_X)\simeq p^r\circ\sm{\tilde{\imath}\circ e_k,f}\simeq p^r\circ\sm{m_k,f}$ which is null-homotopic by assumption. Therefore $h\circ(\Psi_k\wedge\one_X)=h\circ(\Sigma e_k\wedge\one_X)\circ(\psi_k\wedge\one_X)$ is null-homotopic, and by definition of~$\Psi$, the composition
\[
\bigvee^{\infty}_{k=1}\Sigma A^{\wedge k}\wedge X\overset{\Psi\wedge\one}{\longrightarrow}\Sigma\Omega\Sigma A\wedge X\overset{h}{\longrightarrow}BG.
\]
is null-homotopic. Notice that $(\Psi\wedge\one)$ is a homotopy equivalence. It implies that $h$ is null-homotopic and so is $p^r\circ\sm{\tilde{\imath},f}$.
\end{proof}

\begin{lemma}\label{lemma_J_to_JkJl}
The map $p^r\circ\sm{\tilde{\imath}, \tilde{\imath}}$ is null-homotopic if and only if $p^r\circ\sm{m_k, m_l}$ is null-homotopic for all $k$ and $l$.
\end{lemma}

\begin{proof}
It suffices to prove the necessity part. Suppose $p^r\circ\sm{m_k,m_l}$ is null-homotopic for all~$k$ and $l$. Apply the first part of Lemma~\ref{k component of JA} to obtain $p^r\circ\sm{\tilde{\imath},m_l}\simeq*$ for all $l$ and apply the second part of Lemma~\ref{k component of JA} to obtain $p^r\circ\sm{\tilde{\imath},\tilde{\imath}}\simeq*$.
\end{proof}

At this point we have related the order of $\sm{\one_G,\one_G}$ to the orders of $\sm{m_k,m_l}$ for all $k$ and~$l$. There is one more step to link up with the order of the Samelson product
\[
\sm{\imath,\imath}:A\wedge A\overset{\imath\wedge\imath}{\longrightarrow}G\wedge G\overset{\sm{\one_G,\one_G}}{\longrightarrow}G\overset{p^r}{\longrightarrow}G.
\]

\begin{lemma}\label{lemma_sm_mk_null}
If $p^r\circ\sm{\imath,\imath}$ is null-homotopic and $G$ has homotopy nilpotency class less than~\mbox{$p^r+1$}, then $p^r\circ\sm{m_k, m_l}$ is null-homotopic for all $k$ and $l$.
\end{lemma}

The proof of Lemma~\ref{lemma_sm_mk_null} is long and we postpone it to the next section so as to avoid interrupting the flow of our discussion. Assuming Lemma~\ref{lemma_sm_mk_null} we can prove our main theorem.

\begin{thm}\label{main}
Suppose that $G$ has homotopy nilpotence class less than $p^r+1$ after localization at $p$. Then $\sm{\one_G,\one_G}$ has order $p^r$ if and only if $\sm{\imath,\imath}$ has order $p^r$.
\end{thm}

\begin{proof}
The order of $\sm{\one_G, \one_G}$ is not less than the order of $\sm{\imath,\imath}$. Therefore we need to show that the order of $\sm{\one_G, \one_G}$ is not greater than the order of $\sm{\imath,\imath}$ under the assumption. Assume~$\sm{\imath,\imath}$ has order $p^r$, that is $p^r\circ\sm{\imath,\imath}$ is null-homotopic. Lemmas~\ref{rho}, \ref{lemma_J_to_JkJl} and \ref{lemma_sm_mk_null} imply that $p^r\circ\sm{\one_G, \one_G}$ is null-homotopic, so the order of $\sm{\one_G,\one_G}$ is not greater than the order of $\sm{\imath,\imath}$.
\end{proof}

\section{Proof of Lemma~\ref{lemma_sm_mk_null}}
In this section we prove Lemma~\ref{lemma_sm_mk_null} by showing $p^r\circ\sm{m_k,m_l}$ is null-homotopic assuming the homotopy nilpotence class of $G$ is less than $p^r+1$. We convert it into an algebraic problem and derive lemmas from group theoretic identities and the topological properties of~$G$. First let us review the algebraic properties of Samelson products.

\subsection{Algebraic properties of Samelson products}
Given two maps $f:X\to G$ and~\mbox{$g:Y\to G$}, their Samelson product $\sm{f,g}$ sends $(x, y)\in X\wedge Y$ to the commutator of their images $f(x)$ and $g(y)$. It is natural to regard the map $\sm{f,g}$ as a commutator in $[X\wedge Y, G]$, but $f$ and~$g$ are maps in different homotopy sets and there is no direct multiplication between them. Instead, we can include $[X,G],[Y,G]$ and $[X\wedge Y, G]$ into $[X\times Y, G]$ and identify $\sm{f,g}$ as a commutator there.

\begin{lemma}\label{subgp}
For any spaces $X$ and $Y$, let $\pi_1:X\times Y\to X$ and $\pi_2:X\times Y\to Y$ be the projections and let $q:X\times Y\to X\wedge Y$ be the quotient map. Then the images $(\pi_1)^*[X, G]$, $(\pi_2)^*[Y, G]$,  $q^*[X\wedge Y, G]$ are subgroups of $[X \times Y, G]$, and $(\pi_1)^*:[X,G]\to [X\times Y,G]$, $(\pi_2)^*:[Y,G]\to [X\times Y,G]$, $q^*:[X\wedge Y,G]\to [X\times Y,G]$ are monomorphisms.
\end{lemma}

\begin{proof}
Observe that $\pi_1$ and $\pi_2$ induce group homomorphisms $(\pi_1)^*:[X, G]\to[X\times Y, G]$ and~\mbox{$(\pi_2)^*:[X, G]\to[X\times Y, G]$}, so their images $(\pi_1)^*[X, G]$ and $(\pi_2)^*[Y, G]$ are subgroups of $[X\times Y, G]$. Moreover, let $j:X\to X\times Y$ be the inclusion. Since $\pi_1\circ j$ is the identity,~\mbox{$j^*\circ(\pi_1)^*$} is an isomorphism and $(\pi_1)^*$ is a monomorphism. Therefore $[X, G]$ is isomorphic to $(\pi_1)^*[X,G]$. Similarly $[Y, G]$ is isomorphic to $(\pi_2)^*[Y,G]$.

For $q^*[X\wedge Y,G]$, the cofibration $X\vee Y\overset{j'}{\to}X\times Y\overset{q}{\to}X\wedge Y$ induces an exact sequence
\[
\cdots\to[\Sigma X\times Y, G]\overset{\Sigma j'^*}{\longrightarrow}[\Sigma(X\vee Y), G]\longrightarrow[X\wedge Y, G]\overset{q^*}{\longrightarrow}[X\times Y, G]\overset{j'^*}{\longrightarrow}[X\vee Y, G],
\]
where $j'$ is the inclusion and $q$ is the quotient map. Since $\Sigma j':\Sigma(X\vee Y)\hookrightarrow\Sigma(X\times Y)$ has a right homotopy inverse, $q^*:[X\wedge Y, G]\to[X\times Y, G]$ is a monomorphism. Therefore~\mbox{$[X\wedge Y, G]$} is isomorphic to~\mbox{$q^*[X\wedge Y, G]$}, which is a subgroup of $[X\times Y, G]$.
\end{proof}

There are two groups in our discussion, namely $G$ and $[X\times Y, G]$. To distinguish their commutators, for any maps $f:X\to G$ and $g:Y\to G$ we use $\cm{f,g}$ to denote the map which sends $(x,y)\in X\times Y$ to~\mbox{$f(x)^{-1}g(y)^{-1}f(x)g(y)\in G$}, and for any maps $a:X\times Y\to G$ and $b:X\times Y\to G$ we use $[a,b]$ to denote the commutator $a^{-1}b^{-1}ab\in[X\times Y, G]$. Lemma~\ref{subgp} says that $f:X\to G$ and $g:Y\to G$ can be viewed as being in $[X\times Y, G]$. Their images are the compositions
\[
\begin{array}{c c c}
\tilde{f}:X\times Y\overset{\pi_1}{\longrightarrow}X\overset{f}{\longrightarrow}G
&\text{and}
&\tilde{g}:X\times Y\overset{\pi_2}{\longrightarrow}Y\overset{g}{\longrightarrow}G.
\end{array}
\]
Consider the diagram
\[\xymatrix{
																&X\times Y\ar[d]_{f\times g}\ar[r]^{q}\ar[dl]_{\triangle}	&X\wedge Y\ar[d]^{f\wedge g}\\
(X\times Y)\times(X\times Y)\ar[r]^-{\tilde{f}\times\tilde{g}}	&G\times G\ar[d]_{c}\ar[r]^{q'}								&G\wedge G\ar[d]^{\sm{\one_G,\one_G}}\\
																&G\ar[r]^{=}												&G
}\]
where $\triangle$ is the diagonal map, $c$ is the commutator map $\cm{\one_G,\one_G}$, and $q'$ is the quotient maps. The commutativity of the left triangle is due to the definitions of $\tilde{f}$ and $\tilde{g}$, the commutativity of the top square is due to the naturality of the quotient maps and the commutativity of the bottom square is due to the definition of~$\sm{\one_G,\one_G}$. The middle column is $\cm{f,g}$ and the right column is $\sm{f,g}$. In order to show that $\sm{f,g}$ is null-homotopic, it suffices to consider~$\cm{f,g}\simeq q^*\sm{f,g}$ since $q^*:[X\wedge Y, G]\to[X\times Y, G]$ is injective. Observe that $\cm{f,g}$ is homotopic to the composition
\[
X\times Y\overset{\triangle}{\longrightarrow}(X\times Y)\times(X\times Y)\overset{\tilde{f}\times\tilde{g}}{\longrightarrow}G\times G\overset{c}{\longrightarrow}G
\]
according to the diagram. That is $\cm{f,g}$ is the commutator $[\tilde{f},\tilde{g}]=\tilde{f}^{-1}\tilde{g}^{-1}\tilde{f}\tilde{g}$ in $[X\times Y, G]$. 

Let $\ad{a}{b}=b^{-1}ab$ be the conjugation of maps $a$ and $b$ in $[X\times Y, G]$. In group theory, commutators satisfy the following identities:
\begin{equation}\label{lemma_gp_id}
\begin{minipage}{0.9\textwidth}
\begin{enumerate}
\item	$[a,b]=a^{-1}\cdot\ad{a}{b}$;
\item	$[a,b]^{-1}=[b,a]$;
\item	$[a\cdot a',b]=\ad{[a,b]}{a'}\cdot[a',b]=[a,b]\cdot[a',[a,b]]^{-1}\cdot[a',b]$;
\item	$[a,b\cdot b']=[a,b']\cdot\ad{[a,b]}{b'}=[a,b']\cdot[a,b]\cdot[[a,b],b']$.
\end{enumerate}
\end{minipage}
\end{equation}

In particular, we can substitute $\tilde{f}$ and $\tilde{g}$ to $a$ and $b$ in these identities.

\begin{prop}\label{identities}
Let $f,f':X\to G$ and $g,g':Y\to G$ be maps. Then in $[X\times Y, G]$,
\begin{enumerate}[label=\textnormal{(\roman*)}]
\item\label{id_conj}	$\cm{f, g}=\tilde{f}^{-1}\cdot\ad{\tilde{f}}{\tilde{g}}$;
\item\label{id_inv_sm}	$\cm{f, g}^{-1}=\cm{g, f}\circ T$;
\item\label{id_lprod}	$\cm{f\cdot f',g}=\ad{\cm{f, g}}{\tilde{f}'}\cdot\cm{f',g}=\cm{f, g}\cdot[\tilde{f'},[\tilde{f},\tilde{g}]]^{-1}\cdot\cm{f', g}$;
\item\label{id_rprod}	$\cm{f,g\cdot g'}=\cm{f, g'}\cdot\ad{\cm{f,g}}{\tilde{g}'}=c(f,g')\cdot\cm{f, g}\cdot [[\tilde{f}, \tilde{g}],\tilde{g'}]$,
\end{enumerate}
where $T:Y\times X\to X\times Y$ is the swapping map.
\end{prop}

\begin{proof}
All identities come directly from the identities in~(\ref{lemma_gp_id}), while Identity~\ref{id_inv_sm} needs some explanation. Observe there exists a homotopy commutative diagram
\[\xymatrix{
X\times Y\ar[d]^-{T}\ar[r]^-{f\times g}	&G\times G\ar[d]^-{T}\ar[r]^-{c}	&G\ar[d]^-{r}\\
Y\times X\ar[r]^-{g\times f}			&G\times G\ar[r]^-{c}				&G
}\]
where $r:G\to G$ is the inversion. The upper direction around the diagram is $c(f,g)^{-1}$, while the lower direction is $c(g, f)\circ T$. So Identity~\ref{id_inv_sm} follows.
\end{proof}

\begin{remark}
The iterated commutator $[\tilde{f}'(x),[\tilde{f}(x),\tilde{g}(y)]]$ is the composition
\[
X\times Y\overset{\triangle_X\times\one_Y}{\longrightarrow}X\times X\times Y\overset{f'\times f\times g}{\longrightarrow}G\times G\times G\overset{\one_G\times c}{\longrightarrow}G\times G\overset{c}{\longrightarrow}G
\]
where $\triangle_X$ is the diagonal map and $\one_Y$ and $\one_G$ are the identity maps. Let $c_2=c\circ(\one_G\times c)$ be the 2-iterated commutator on $G$. Then we can write $[\tilde{f}',[\tilde{f},\tilde{g}]]$ as $c_2\circ(f'\times f\times g)\circ(\triangle_X\times\one_Y)$. However, we prefer to stick to the notation $[\tilde{f}',[\tilde{f},\tilde{g}]]$ because it better indicates it is the commutator of which maps, while $c_2\circ(f'\times f\times g)\circ(\triangle_X\times\one_Y)$ looks long and confusing.
\end{remark}

Since our group $[X\times Y, G]$ has a topological interpretation, the topologies of $X$ and $Y$ add extra algebraic properties to its group structure.

\begin{lemma}\label{lemma_co-H_id}
Let $f, g$ and $h: X\times Y\to G$ be maps. If $X$ is a co-H-space and the restrictions of $f$ and $g$ to $X\vee Y$ are null-homotopic, then in $[X\times Y, G]$ we have
\[
\begin{array}{c c c}
f\cdot g=g\cdot f
&\text{and}
&[f\cdot g, h]=[f, h]\cdot[g, h].
\end{array}
\]
\end{lemma}

\begin{proof}
Let $q:X\times Y\to X\wedge Y$ be the quotient map. Observe that there exist $f'$ and $g'$ in~$[X\wedge Y, G]$ such that $f=q^*f'$ and $g=q^*g'$. Since $X\wedge Y$ is a co-H-space, $[X\wedge Y, G]$ is an abelian group and $f'$ and $g'$ commute. Therefore $f$ and $g$ commute as $q^*$ is a monomorphism.

To show the linearity, we start with Proposition~\ref{identities}~\ref{id_lprod}
\[
[f\cdot g, h]=[f, h]\cdot[g, [f, h]]\cdot[g, h].
\]
Since $[f, h]$ is also null-homotopic on $X\vee Y$, it commutes with $g$ and their commutator~$[g, [f,h]]$ is trivial. Therefore we have $[f\cdot g, h]=[f,h]\cdot[g,h]$.
\end{proof}

\subsection{Main body of the proof}
We go back to the proof of Lemma~\ref{lemma_sm_mk_null}. Recall that $m_k$ is the composition $m_k:A^k\overset{\imath^k}{\longrightarrow}G^k\overset{m^k}{\longrightarrow}G$. To distinguish the spaces $A$'s, denote the $i^{\text{th}}$ copy of $A$ in $A^k$ by $A_i$. Let $a_i$ and $m'_{k-1}$ be the compositions
\[
\begin{array}{c c c}
a_i:A^k\overset{proj}{\longrightarrow}A_i\overset{\imath}{\longrightarrow}G
&\text{and}
&m'_{k-1}:A^k\overset{proj}{\longrightarrow}\prod^{k-1}_{i=1}A_i\overset{m_{k-1}}{\longrightarrow}G
\end{array}
\]
respectively. Then we have $m_k=m'_{k-1}\cdot a_i$ in $[A^k, G]$. Include $\sm{m_k,m_l}$ in $[A^k\times A^l, G]$ by Lemma~\ref{subgp}. It becomes the commutator $\cm{m_k, m_l}=[\tilde{m}_k, \tilde{m}_l]$, where $\tilde{m}_k$ and $\tilde{m}_l$ are compositions
\[
\begin{array}{c c c}
\tilde{m}_k:A^k\times A^l\overset{proj}{\longrightarrow}A^k\overset{m_k}{\longrightarrow}G
&\text{and}
&\tilde{m}_l:A^k\times A^l\overset{proj}{\longrightarrow}A^l\overset{m_l}{\longrightarrow}G.
\end{array}
\]
Let $\tilde{a}_i$ and $\tilde{m}'_{k-1}$ be compositions
\[
\begin{array}{c c c}
\tilde{a}_i:A^k\times A^l\overset{proj}{\longrightarrow}A_i\overset{\imath}{\longrightarrow}G
&\text{and}
&\tilde{m}'_{k-1}:A^k\times A^l\overset{proj}{\longrightarrow}\prod^{k-1}_{i=1}A_i\overset{m_{k-1}}{\longrightarrow}G.
\end{array}
\]
Then in $[A^k\times A^l, G]$ we have $\tilde{m}_k=\tilde{m}'_{k-1}\cdot\tilde{a}_k$.

Assume the homotopy nilpotence class of $G$ is less than $p^r+1$. Now we use induction on~$k$ and $l$ show that $\cm{m_k,m_l}^{p^r}$ is null-homotopic. To start with, we show that this is true for~$k=1$ or $l=1$.

\begin{lemma}\label{induction_1}
If $\cm{\imath,\imath}^{p^r}$ is null-homotopic, then $\cm{m_k,\imath}^{p^r}$ and $\cm{\imath,m_l}^{p^r}$ are null-homotopic for all $k$ and $l$.
\end{lemma}

\begin{proof}
We prove that $\cm{m_k,\imath}^{p^r}$ is null-homotopic by induction. Since $m_1=\imath$,~\mbox{$\cm{m_1,\imath}^{p^r}=\cm{\imath,\imath}^{p^r}$} is null-homotopic by assumption. Suppose $\cm{m_k,\imath}^{p^r}$ is null-homotopic. We need to show that $\cm{m_{k+1},\imath}^{p^r}$ is also null-homotopic. Apply Proposition~\ref{identities}~\ref{id_lprod} to obtain
\[
\cm{m_{k+1},\imath}=\cm{m'_k\cdot a_{k+1},\imath}=\ad{\cm{m'_k,\imath}}{\tilde{a}_{k+1}}\cdot\cm{a_{k+1},\imath}.
\]
Observe that $\cm{a_{k+1},\imath}$ and $\ad{\cm{m'_k,\imath}}{\tilde{a}_{k+1}}$ are null-homotopic on $A^{k+1}\vee A$ and $A^{k+1}\wedge A$ is a co-H-space. Lemma~\ref{lemma_co-H_id} implies that $\cm{a_{k+1},\imath}$ and $\ad{\cm{m'_k,\imath}}{\tilde{a}_{k+1}}$ commute and we have
\begin{eqnarray*}
\cm{m_{k+1},\imath}^{p^r}
&=&\left(\ad{\cm{m'_k,\imath}}{\tilde{a}_{k+1}}\cdot\cm{a_{k+1},\imath}\right)^{p^r}\\
&=&\left(\ad{\cm{m'_k,\imath}}{\tilde{a}_{k+1}}\right)^{p^r}\cdot\cm{a_{k+1},\imath}^{p^r}\\
&=&\ad{\left(\cm{m'_k,\imath}^{p^r}\right)}{\tilde{a}_{k+1}}\cdot\cm{a_{k+1},\imath}^{p^r}.
\end{eqnarray*}
The last term $\cm{a_{k+1},\imath}^{p^r}$ is null-homotopic since $a_{k+1}$ is the inclusion $A_{k+1}\overset{\imath}{\to}G$. Also, by the induction hypothesis $\cm{m'_k,\imath}^{p^r}$ is null-homotopic. Therefore $\cm{m_{k+1},\imath}^{p^r}$ is null-homotopic and the induction is completed.

Similarly, we can show that $\cm{\imath, m_l}^{p^r}$ is null-homotopic for all $l$.
\end{proof}

As a consequence of Lemma~\ref{induction_1}, the following lemma implies that the order of
\[
\sm{\imath, \one_G}:A\wedge G\overset{\imath\wedge\one_G}{\longrightarrow}G\wedge G\overset{\sm{\one_G,\one_G}}{\longrightarrow}G
\]
equals to the order of its restriction $\sm{\imath,\imath}$ without assuming the condition on the homotopy nilpotence of $G$. 

\begin{lemma}\label{A wedge G}
The map $p^r\circ\sm{\imath,\imath}$ is null-homotopic if and only if $p^r\circ\sm{\one_G,\imath}$ and $p^r\circ\sm{\imath,\one_G}$ are null-homotopic.
\end{lemma}

\begin{proof}
We only need to prove the sufficient condition. If $p^r\circ\sm{\imath,\imath}$ is null-homotopic, then~\mbox{$p^r\circ\sm{\imath,m_l}$} is null-homotopic for all~$l$ by Lemma~\ref{induction_1}. Lemma~\ref{k component of JA} implies that $p^r\circ\sm{\imath,\tilde{\imath}}:A\wedge\Omega\Sigma A\to G$ is null-homotopic. Since $\tilde{\imath}\circ s$ is a homotopy equivalence by Lemma~\ref{lemma_rho+s}, $p^r\circ\sm{\imath,\one_G}$ is null-homotopic.

The sufficient condition for $p^r\circ\sm{\one_G,\imath}$ can be proved similarly.
\end{proof}

Now suppose $\cm{m_k,m_l}^{p^r}$ is trivial for some fixed $k$ and $l$ in $[A^k\times A^l, G]$. The next step is to show that $\cm{m_{k+1},m_l}^{p^r}$ is trivial in $[A^{k+1}\times A^l, G]$. At first glance we can follow the proof of Lemma~\ref{induction_1} and apply Lemmas~\ref{identities} and~\ref{lemma_co-H_id} to split $\cm{m_{k+1},m_l}^{p^r}$ into $\cm{m'_k,m_l}^{p^r}$ and~\mbox{$\cm{a_{k+1},m_l}^{p^r}$} which are null-homotopic by the induction hypothesis. However, when $l>1$,~$A^l$ is not a co-H-space and we cannot use Lemma~\ref{lemma_co-H_id} to argue that $\cm{m'_k,m_l}$ and $\cm{a_{k+1},m_l}$ commute. Instead, apply Proposition~\ref{identities}~\ref{id_lprod} to obtain
\begin{eqnarray*}
\cm{m_{k+1},m_l}
&=&\cm{m'_k\cdot a_{k+1}, m_l}\\
&=&\cm{m'_k,m_l}\cdot[\cm{m'_k,m_l},\tilde{a}_{k+1}]\cdot\cm{a_{k+1},m_l}.
\end{eqnarray*}
Denote $\cm{m'_k,m_l}$ and $[\cm{m'_k,m_l},\tilde{a}_{k+1}]\cdot\cm{a_{k+1},m_l}$ by $\alpha_k$ and $\beta_k$ respectively. Observe that the restrictions of any powers and commutators involving $\beta_k$ to $A_{k+1}\vee(A^k\times A^l)$ are null-homotopic and $A_{k+1}$ is a co-H-space. Therefore they enjoy the conditions of Lemma~\ref{lemma_co-H_id}.

\begin{lemma}\label{alpha_k+1^n}
For any natural number $n$, we have
\[
(\alpha_k\cdot\beta_k)^n=\alpha_k^n\cdot\beta_k^n\cdot\left(\prod^{n-1}_{i=1}[\beta_k,\alpha_k^i]\right).
\]
\end{lemma}

\begin{proof}
We induct on $n$. The statement of the lemma is trivial for $n=1$. Assume the formula holds for an integer $n$. For the $(n+1)$ case, using the induction hypothesis we have
\begin{eqnarray*}
(\alpha_k\cdot\beta_k)^{n+1}
&=&\alpha_k\cdot\beta_k\cdot(\alpha_k\cdot\beta_k)^n\\
&=&\alpha_k\cdot\beta_k\cdot\alpha_k^n\cdot\beta_k^n\cdot\left(\prod^{n-1}_{i=1}[\beta_k,\alpha_k^i]\right)\\
&=&\alpha_k^{n+1}\cdot\beta_k\cdot[\beta_k,\alpha_k^n]\cdot\beta_k^n\cdot\left(\prod^{n-1}_{i=1}[\beta_k,\alpha_k^i]\right)
\end{eqnarray*} 
In the last line $[\beta_k,\alpha_k^n]$ is formed after we swap $\beta_k$ and $\alpha_k^n$. Since the restrictions of $\beta_k$,~$\beta^n_k$ and~$[\beta_k,\alpha^i_k]$ to $A_{k+1}\vee(A^k\times A^l)$ are null-homotopic, they commute by Lemma~\ref{lemma_co-H_id}. By commuting the terms, the statement follows.
\end{proof}

In order to prove the triviality of $\cm{m_{k+1},m_l}^{p^r}$, by Lemma~\ref{alpha_k+1^n} it suffices to show that~\mbox{$\alpha_k^{p^r},\beta_k^{p^r}$} and~\mbox{$\prod^{p^r-1}_{i=1}[\beta_k,\alpha_k^i]$} are null-homotopic. By the induction hypothesis $\cm{m_k,m_l}^{p^r}$ is null-homotopic, so $\alpha_k^{p^r}=\cm{m'_k,m_l}^{p^r}$ is null-homotopic. It remains to show that $\beta_k^{p^r}$ and $\prod^{p^r-1}_{i=1}[\beta_k,\alpha_k^i]$ are null-homotopic.

\begin{lemma}\label{alpha_beta}
If $\cm{\imath,\imath}^{p^r}$ and $\alpha_k^{p^r}$ are null-homotopic, then so is $\beta_k^{p^r}$.
\end{lemma}

\begin{proof}
By definition, $\beta_k=[\cm{m'_k,m_l},\tilde{a}_{k+1}]\cdot\cm{a_{k+1},m_l}$. Observe that the restrictions of~$[\cm{m'_k,m_l},\tilde{a}_{k+1}]$ and~\mbox{$\cm{a_{k+1},m_l}$} to $A_{k+1}\vee(A^k\times A^l)$ are null-homotopic. By Lemma~\ref{lemma_co-H_id} they commute and we have
\[
\beta_k^{p^r}=\left([\cm{m'_k,m_l},\tilde{a}_{k+1}]\cdot\cm{a_{k+1},m_l}\right)^{p^r}=[\cm{m'_k,m_l},\tilde{a}_{k+1}]^{p^r}\cdot\cm{a_{k+1},m_l}^{p^r}.
\]
Since $\cm{\imath,\imath}^{p^r}$ is null-homotopic, so is $\cm{a_{k+1},m_l}^{p^r}$ by Lemma~\ref{induction_1}.

On the other hand, recall that $c(m'_k,m_l)$ and $\tilde{a}_{k+1}$ are the compositions
\[
\begin{array}{c c c}
c(m'_k,m_l):A^{k+1}\times A^l\overset{proj}{\longrightarrow}A^k\times A^l\overset{c(m_k,m_l)}{\longrightarrow}G
&\text{and}
&\tilde{a}_{k+1}:A^{k+1}\times A^l\overset{proj}{\longrightarrow}A_{k+1}\overset{\imath}{\longrightarrow}G
\end{array}
\]
respectively. Therefore we have
\begin{eqnarray*}
[\cm{m'_k,m_l}, \tilde{a}_{k+1}]^{p^r}
&=&p^r\circ c(c(m'_k,m_l),\imath)\\
&=&p^r\circ\cm{\one_G,\imath}\circ(c(m'_k,m_l)\times\one_A)
\end{eqnarray*}
where $\one_A$ is the identity map of $A_{k+1}$. Since $p^r\circ\cm{\one_G,\imath}$ is null-homotopic by Lemma~\ref{A wedge G},~\mbox{$[\cm{m'_k,m_l}, \tilde{a}_{k+1}]^{p^r}$} is null-homotopic and so is $\beta^{p^r}_k$.
\end{proof}

\begin{lemma}\label{prod_express}
For any natural number $n$, we have
\[
\prod^{n-1}_{i=1}[\beta_k,\alpha_k^i]=\prod^{n-1}_{i=1}c_i(\beta_k,\alpha_k,\cdots,\alpha_k)^{\binom{n}{i+1}}
\]
where $c_i(\beta_k,\alpha_k,\cdots,\alpha_k)=[[\cdots[[\beta_k,\alpha_k],\alpha_k]\cdots], \alpha_k]$ is the $i$-iterated commutator.
\end{lemma}

\begin{proof}
First, by induction we prove
\[
[\beta_k,\alpha^i_k]=\prod^i_{j=1}c_j(\beta_k,\alpha_k,\cdots,\alpha_k)^{\binom{i}{j}}.
\]
It is trivial for $i=1$. Assume the formula holds for $[\beta_k,\alpha^i_k]$. Use the commutator identity in~(\ref{lemma_gp_id}) and inductive hypothesis to get
\begin{eqnarray*}
[\beta_k,\alpha^{i+1}_k]
&=&[\beta_k,\alpha_k]\cdot[\beta_k,\alpha^i_k]\cdot[[\beta_k,\alpha^i_k],\alpha_k]\\
&=&[\beta_k,\alpha_k]\cdot\left(\prod^i_{j=1}c_j(\beta_k,\alpha_k,\cdots,\alpha_k)^{\binom{i}{j}}\right)\cdot[\prod^i_{j=1}c_j(\beta_k,\alpha_k,\cdots,\alpha_k)^{\binom{i}{j}},\alpha_k]
\end{eqnarray*}

Since the restriction of $c_j(\beta_k,\alpha_k,\cdots,\alpha_k)$ to $A_{k+1}\vee(A^k\times A^l)$ is null-homotopic for all $j$, by Lemma~\ref{lemma_co-H_id} they commute with each other and
\begin{eqnarray*}
[\beta_k,\alpha^{i+1}_k]
&=&[\beta_k,\alpha_k]\cdot\left(\prod^i_{j=1}c_j(\beta_k,\alpha_k,\cdots,\alpha_k)^{\binom{i}{j}}\right)\cdot\left(\prod^i_{j=1}[c_j(\beta_k,\alpha_k,\cdots,\alpha_k),\alpha_k]^{\binom{i}{j}}\right)\\
&=&[\beta_k,\alpha_k]\cdot\left(\prod^i_{j=1}c_j(\beta_k,\alpha_k,\cdots,\alpha_k)^{\binom{i}{j}}\right)\cdot\left(\prod^i_{j=1}c_{j+1}(\beta_k,\alpha_k,\cdots,\alpha_k)^{\binom{i}{j}}\right)\\
&=&[\beta_k,\alpha_k]^{i+1}\cdot\left(\prod^i_{j=2}c_j(\beta_k,\alpha_k,\cdots,\alpha_k)^{\binom{i}{j}+\binom{i}{j+1}}\right)\cdot c_{i+1}(\beta_k,\alpha_k,\cdots,\alpha_k)\\
&=&\prod^{i+1}_{j=1}c_j(\beta_k,\alpha_k,\cdots,\alpha_k)^{\binom{i+1}{j}}
\end{eqnarray*}
Therefore the claim is proved.

Now we multiply all $[\beta_k,\alpha^i_k]$'s and use the commutativity of $c_j(\beta_k,\alpha_k,\cdots,\alpha_k)$'s to get
\begin{eqnarray*}
\prod^{n-1}_{i=1}[\beta_k,\alpha^i_k]
&=&\prod^{n-1}_{i=1}\prod^i_{j=1}c_j(\beta_k,\alpha_k,\cdots,\alpha_k)^{\binom{i}{j}}\\
&=&\prod^{n-1}_{j=1}\prod^{n-1}_{i=j}c_j(\beta_k,\alpha_k,\cdots,\alpha_k)^{\binom{i}{j}}\\
&=&\prod^{n-1}_{j=1}\left(\prod^{n-1}_{i=j}c_j(\beta_k,\alpha_k,\cdots,\alpha_k)^{\binom{i}{j}}\right)\\
&=&\prod^{n-1}_{j=1}c_j(\beta_k,\alpha_k,\cdots,\alpha_k)^{\sum^{n-1}_{i=j}\binom{i}{j}}
\end{eqnarray*}
The proof will be completed if we can show that $\sum^{n-1}_{i=j}\binom{i}{j}=\binom{n}{j+1}$.

Consider the polynomial
\[
\sum^{n-1}_{i=0}(1+x)^i=\sum^{n-1}_{i=0}\sum^i_{j=0}\binom{i}{j}x^j=\sum^{n-1}_{j=0}\sum^{n-1}_{i=0}\binom{i}{j}x^j.
\]
The coefficient of $x^j$ is $\sum^{n-1}_{i=j}\binom{i}{j}$. On the other hand, it can be written as
\[
1+(1+x)+\cdots+(1+x)^{n-1}=\frac{(1+x)^n-1}{x}=\sum^{n}_{j=1}\binom{n}{j}x^{j-1}.
\]
The coefficient of $x^j$ is $\binom{n}{j+1}$. By comparing the coefficients of $x^j$ the statement follows.
\end{proof}

Let $nil(G)$ be the homotopy nilpotence class of $G$, that is, $nil(G)=n$ if and only if the~$n$-iterated commutator $c_n$ is null-homotopic but $c_{n-1}$ is not. In particular, $nil(G)=n$ means all commutators in $G$ of length greater than $n$ are null-homotopic.

\begin{lemma}\label{nil_omega}
Let $\{f_j:A^k\times A^l\to G\}_{1\leq j\leq n-1}$ be a set of maps. If $nil(G)\leq n$, then~\mbox{$c_{n-1}(\beta_k,f_1,\cdots,f_{n-1})$} is null-homotopic.
\end{lemma}

\begin{proof}
By definition, $\beta_k=[\cm{m'_k,m_l},\tilde{a}_{k+1}]\cdot\cm{a_{k+1},m_l}$. Denote $[\cm{m'_k,m_l},\tilde{a}_{k+1}]$ and $\cm{a_{k+1},m_l}$ by $\gamma_0$ and $\gamma'_0$ respectively. Then we have $\beta_k=\gamma_0\cdot\gamma'_0$. For $1\leq j\leq n-1$, let $\gamma_j=[\gamma_{j-1},f_j]$ and $\gamma'_j=[\gamma'_{j-1},f_j]$. We claim that $c_m(\beta_k,f_1,\cdots,f_m)=\gamma_m\cdot\gamma'_m$ for $1\leq m\leq n-1$.

When $m=1$,
\[
c_1(\beta_k,f_1)=[\beta_k,f_1]=[\gamma_0\cdot\gamma'_0,f_1].
\]
Since the restrictions of $\gamma_0$ and $\gamma'_0$ to $A_{k+1}\vee(A^k\times A^l)$ are null-homotopic, by Lemma~\ref{lemma_co-H_id} we have
\[
[\gamma_0\cdot\gamma'_0,f_1]=[\gamma_0,f_1]\cdot[\gamma'_0,f_1]=\gamma_1\cdot\gamma'_1.
\]
Assume the claim is true for $m-1$. By the induction hypothesis,
\begin{eqnarray*}
c_m(\beta_k,f_1,\cdots,f_m)
&=&c_1\circ(c_{m-1}(\beta_k, f_1,\cdots, f_{m-1})\times f_m)\\
&=&[c_{m-1}(\beta_k, f_1,\cdots, f_{m-1}), f_m]\\
&=&[\gamma_{m-1}\cdot\gamma'_{m-1}, f_m]
\end{eqnarray*}
Since the restrictions of $\gamma_{m-1}$ and $\gamma'_{m-1}$ to $A_{k+1}\vee(A^k\times A^l)$ are null-homotopic, by Lemma~\ref{lemma_co-H_id}
\[
c_m(\beta_k,f_1,\cdots,f_m)=[\gamma_{m-1},f_m]\cdot[\gamma'_{m-1}, f_m]=\gamma_m\cdot\gamma'_m.
\]

By putting $m=n-1$ we get $c_{n-1}(\beta_k,f_1,\cdots,f_n)=\gamma_{n-1}\cdot\gamma'_{n-1}$. Notice that $\gamma_{n-1}$ and~$\gamma'_{n-1}$ are commutators of length $n+2$ and $n+1$ respectively, which are null-homotopic due to the condition on the homotopy nilpotency of $G$. Therefore $c_{n-1}(\beta_k,f_1,\cdots,f_{n-1})$ is null-homotopic.
\end{proof}

Now we have all the ingredients to prove Lemma~\ref{lemma_sm_mk_null}.

\begin{proof}[Proof of Lemma~\ref{lemma_sm_mk_null}]
Suppose $\cm{\imath,\imath}^{p^r}$ is null-homotopic and $nil(G)$ is less than $p^r+1$. We prove that $\cm{m_k, m_l}^{p^r}$ is null-homotopic for all $k$ and $l$ by induction. By Lemma~\ref{induction_1},~$\cm{m_k,\imath}^{p^r}$ and $\cm{\imath, m_l}^{p^r}$ are null-homotopic for all $k$ and $l$. Assume $\cm{m_k,m_l}^{p^r}$ is null-homotopic for some fixed $k$ and $l$. We need to show that $\cm{m_{k+1},m_l}^{p^r}$ is null-homotopic. By Lemma~\ref{alpha_k+1^n}, we have
\[
\cm{m_{k+1},m_l}^{p^r}=\alpha_k^{p^r}\cdot\beta_k^{p^r}\cdot\prod^{p^r-1}_{i=1}[\beta_k,\alpha_k^i].
\]
The factor $\alpha_k^{p^r}$ is null-homotopic due to the definition of $\alpha_k$ and hypothesis, and $\beta_k^{p^r}$ is null-homotopic by Lemma~\ref{alpha_beta}. So it remains to show that $\prod^{p^r-1}_{i=1}[\beta_k,\alpha_k^i]$ is null-homotopic.

By Lemma~\ref{prod_express},
\[
\prod^{p^r-1}_{i=1}[\beta_k,\alpha_k^i]=\prod^{p^r-1}_{j=1}c_j(\beta_k,\alpha_k,\cdots,\alpha_k)^{\binom{p^r}{j+1}}
\]
Observe that $\binom{p^r}{j+1}$ is divisible by $p^r$ for $1\leq j\leq p^r-2$. By Lemma~\ref{lemma_co-H_id} we have
\[
c_j(\beta_k^n,\alpha_k,\cdots,\alpha_k)=c_j(\beta_k,\alpha_k,\cdots,\alpha_k)^n
\]
for all $n$. In our case we have
\[
c_j(\beta_k,\alpha_k,\cdots,\alpha_k)^{\binom{p^r}{j+1}}=c_j(\beta_k^{p^r},\alpha_k,\cdots,\alpha_k)^{\binom{p^r}{j+1}/p^r}.
\]
Also, when $j=p^r-1$, the term $c_{p^r-1}(\beta_k,\alpha_k,\cdots,\alpha_k)$ is null-homotopic by Lemma~\ref{nil_omega}. Putting these together we obtain
\[
\prod^{p^r-1}_{i=1}[\beta_k,\alpha_k^i]=\prod^{p^r-2}_{j=1}c_j(\beta_k^{p^r},\alpha_k,\cdots,\alpha_k)^{\binom{p^r}{j+1}/p^r}.
\]
We have shown that $\beta_k^{p^r}$ is null-homotopic in Lemma~\ref{alpha_beta}, so $\prod^{p^r}_{i=1}[\beta_k,\alpha_k^i]$ is null-homotopic and the induction is completed.
\end{proof}

\section{Orders of Samelson products of quasi-$p$-regular groups}
In this section we apply Theorem~\ref{main} to calculate the orders of $\sm{\one_G,\one_G}$ for certain Lie groups~$G$. Recall that $G$ is rationally homotopy equivalent to a product of spheres~$\prod^l_{i=1}S^{2n_i-1}$, where $n_1\leq\cdots\leq n_l$. The sequence $(2n_1-1,\cdots,2n_l-1)$ is called the \emph{type} of $G$. After localization at $p$, $G$ is homotopy equivalent to a product of H-spaces $\prod^{p-1}_{i=1}B_i$, and $A$ is homotopy equivalent to a wedge of co-H-spaces $\bigvee^{p-1}_{i=1}A_i$ such that $A_i$ is a subspace of $B_i$. For~\mbox{$1\leq i\leq p-1$}, let~\mbox{$\imath_i:A_i\to B_i$} be the inclusion. Then $H_*(B_i)$ is the exterior algebra generated by~$(\imath_i)_*(\tilde{H}_*(A_i))$. If each~$B_i$ is a sphere, then we call $G$ \emph{$p$-regular}. If each $A_i$ is a sphere or a CW-complex with two cells, then we call $G$ \emph{quasi-$p$-regular}. When $A_i$ is a CW-complex with two cells, it is homotopy equivalent to the cofibre of $\alpha_{2n_i-1}$, which is the generator of the homotopy group~$\pi_{2n_i+2p-4}(S^{2n_i-1})$, and the corresponding $B_i$ is the~$S^{2n_i-1}$-bundle~$B(2n-1, 2n+2p-3)$ over $S^{2n_i+2p-3}$ classified by $\frac{1}{2}\alpha_{2n_i-1}$ \cite{MNT77}.

The homotopy nilpotence classes of certain quasi-$p$-regular Lie groups are known.
\begin{thm}[Kaji and Kishimoto \cite{kk10}]
A $p$-regular Lie group has homotopy nilpotence class at most 3.
\end{thm}

\begin{thm}[Kishimoto \cite{kishimoto09}]
For $p\geq7$, a quasi-$p$-regular $SU(n)$ has homotopy nilpotence class at most 3.
\end{thm}

\begin{thm}[Theriault \cite{theriault16}]
For $p\geq7$, a quasi-$p$-regular exceptional Lie group has homotopy nilpotence class at most 2.
\end{thm}

For $t=n-p+1$ and $t'=n-\frac{1}{2}p+1$, assume $G$ and $p$ are in the following list:
\begin{equation}\label{quasi_list}
\begin{array}{l l l}
SU(n)	&\simeq B(3,2p+1)\times\cdots\times B(2t-1,2n-1)\times S^{2t+1}\times\cdots\times S^{2p-1}	&p>\frac{n}{2}\\
		&\simeq S^3\times S^5\times\cdots\times S^{2n-1}	&n\leq p\leq\frac{n}{2}\\
Sp(n)	&\simeq B(3,2p+1)\times\cdots\times B(2t'-1,4n-1)\times S^{2t'+1}\times\cdots\times S^{2p-1}	&p>n\\
F_4		&\simeq B(3,15)\times B(11,23)	&p=7\\
		&\simeq B(3,23)\times S^{11}\times S^{15}	&p=11\\
		&\simeq S^3\times S^{11}\times S^{15}\times S^{23}	&p>11\\
E_6		&\simeq F_4\times S^9\times S^{17}	&p\geq7\\
E_7		&\simeq B(3, 23)\times B(15,35)\times S^{11}\times S^{19}\times S^{27}	&p=11\\
		&\simeq B(3, 27)\times B(11,35)\times S^{15}\times S^{19}\times S^{23}	&p=13\\
		&\simeq B(3, 35)\times S^{11}\times S^{15}\times S^{19}\times S^{23}\times S^{27}	&p=17\\
		&\simeq S^3\times S^{11}\times S^{15}\times S^{19}\times S^{23}\times S^{27}\times S^{35}	&p>17\\
E_8		&\simeq B(3, 23)\times B(15,35)\times B(27,47)\times B(39,59)	&p=11\\
		&\simeq B(3, 27)\times B(15,39)\times B(23,47)\times B(35,59)	&p=13\\
		&\simeq B(3, 35)\times B(15,47)\times B(27,59)\times S^{23}\times S^{39}	&p=17\\
		&\simeq B(3, 39)\times B(23,59)\times S^{15}\times S^{27}\times S^{35}\times S^{47}	&p=23\\
		&\simeq B(3, 59)\times S^{15}\times S^{23}\times S^{27}\times S^{35}\times S^{39}\times S^{47}	&p=29\\
		&\simeq S^3\times S^{15}\times S^{23}\times S^{27}\times S^{35}\times S^{39}\times S^{47}\times S^{59}	&p>29.
\end{array}
\end{equation}
By Theorem~\ref{main}, the order of $\sm{\one_G,\one_G}$ equals the order of $\sm{\imath,\imath}$ in these groups.

\subsection{Upper bounds on the orders of $\sm{\one_G,\one_G}$ for quasi-$p$-regular Lie groups}
Since $\sm{\imath,\imath}\in[A\wedge A,G]$ and
\begin{equation}\label{A_wedge_A_decomp}
[A\wedge A,G]\cong[(\bigvee^{p-1}_{i=1}A_i)\wedge(\bigvee^{p-1}_{j=1}A_j), G]\cong\prod^{p-1}_{i,j=1}[A_i\wedge A_j, G]\cong\prod^{p-1}_{i,j,k=1}[A_i\wedge A_j, B_k],
\end{equation}
the order of $\sm{\imath,\imath}$ cannot exceed the least common multiple of the orders of $[A_i\wedge A_j, B_k]$ for all $i,j$ and $k$. Let $C_{2n_i-1}$ be the cofiber of the generator $\alpha_{2n_i-1}$ of the homotopy group~$\pi_{2n_i+2p-4}(S^{2n_i-1})$. When $G$ is quasi-$p$-regular, each $A_i$ is a sphere $S^{2n_i-1}$ or $C_{2n_i-1}$, so~$A_i\wedge A_j$ is either $S^{2n_i+2n_j-2}$, $C_{2n_i+2n_j-2}$ or $C_{2i-1}\wedge C_{2n_j-1}$. In the following we consider the orders of $[A_i\wedge A_j,B_k]$ case by case.

If $A_i\wedge A_j$ is $S^{2n_i+2n_j-2}$, then $[A_i\wedge A_j, B_k]$ is $\pi_{2n_i+2n_j-2}(B_k)$. The homotopy groups of $B_k$ are known in a range.

\begin{thm}[Toda \cite{toda62}, Mimura and Toda \cite{MT70}, Kishimoto \cite{kishimoto09}]\label{toda}
Localized at $p$, we have
\[
\pi_{2n-1+k}(S^{2n-1})\cong\begin{cases}
\Z/p\Z	&\text{for $k=2i(p-1)-1$, $1\leq i\leq p-1$}\\
\Z/p\Z	&\text{for $k=2i(p-1)-2$, $n\leq i\leq p-1$}\\
0		&\text{other cases for $1\leq k\leq 2p(p-1)-3$},
\end{cases}
\]
\[
\pi_{2n-1+k}(B(3, 2p+1))\cong\begin{cases}
\Z/p\Z	&\text{for $k=2i(p-1)-1$, $2\leq i\leq p-1$}\\
\Z		&\text{for $k=2p-2$}\\
0		&\text{other cases for $1\leq k\leq 2p(p-1)-3$},
\end{cases}
\]
and
\[
\pi_{2n-1+k}(B(2n-1, 2n+2p-3))\cong\begin{cases}
\Z/p^2\Z	&\text{for $k=2i(p-1)-1$, $2\leq i\leq p-1$}\\
\Z/p\Z		&\text{for $k=2i(p-1)-2$, $n\leq i\leq p-1$}\\
\Z			&\text{for $k=2p-2$}\\
0			&\text{other cases for $1\leq k\leq 2p(p-1)-3$}.
\end{cases}
\]
\end{thm}
\noindent
Since $2n_i+2n_j-2$ is even, $\pi_{2n_i+2n_j-2}(B_k)$ is isomorphic to either $0$, $\Z/p\Z$ or $\Z/p^2\Z$. Therefore the order of $[A_i\wedge A_j, B_k]$ is at most $p^2$.

If $A_i\wedge A_j$ is $C_{2n_i+2n_j-2}$, then the cofibration
\[
S^{2n_i+2n_j-2}\to C_{2n_i+2n_j-2}\to S^{2n_i+2n_j+2p-4}
\]
induces an exact sequence
\begin{equation}\label{dgm_exp_SES_SC}
\pi_{2n_i+2n_j+2p-4}(B_k)\to[C_{2n_i+2n_j-2}, B_k]\to\pi_{2n_i+2n_j-2}(B_k).
\end{equation}
Since $C_{2n_i+2n_j-2}$ is a suspension and $B_k$ is an H-space, the three groups are abelian. By Theorem~\ref{toda}, the first and the last homotopy groups have orders at most $p^2$, so the order of~$[C_{2n_i+2n_j-2},B_k]$ is at most $p^4$.

If $A_i\wedge A_j$ is $C_{2n_i-1}\wedge C_{2n_j-1}$, then it is a CW-complex with one cell of dimension~\mbox{$2n_i+2n_j-2$}, two cells of dimension $2n_i+2n_j+2p-4$ and one cell of dimension~\mbox{$2n_i+2n_j+4p-6$}. Let~$C'$ be the $(2n_i+2n_j+4p-7)$-skeleton of $C_{2n_i-1}\wedge C_{2n_j-1}$, that is, $C_{2n_i-1}\wedge C_{2n_j-1}$ minus the top cell. Then the cofibration $C'\to C_{2n_i-1}\wedge C_{2n_j-1}\to S^{2n_i+2n_j+4p-6}$ induces an exact sequence of abelian groups
\[
\pi_{2n_i+2n_j+4p-6}(B_k)\longrightarrow[C_{2n_i-1}\wedge C_{2n_j-1},B_k]\longrightarrow[C',B_k].
\]
According to \cite{gray98}, $C'$ is homotopy equivalent to $C_{2n_i+2n_j-2}\vee S^{2n_i+2n_j+2p-4}$, so we have
\[
[C', B_k]\cong[C_{2n_i+2n_j-2}, B_k]\oplus\pi_{2n_i+2n_j+2p-4}(B_k).
\]
We have shown that $[C_{2n_i+2n_j-2}, B_k]$ has order at most $p^4$. By Theorem~\ref{toda}, $\pi_{2n_i+2n_j+2p-4}(B_k)$ and $\pi_{2n_i+2n_j+4p-6}(B_k)$ have orders at most $p^2$. Therefore the order of $[C_{2n_i-1}\wedge C_{2n_j-1}, B_k]$ is at most~$p^6$.

Summarizing the above discussion, we have the following proposition.

\begin{prop}\label{bound_general}
Let $G$ and $p$ be in (\ref{quasi_list}). Then the order of $\sm{\one_G,\one_G}$ is at most $p^6$.
\end{prop}

This gives a very rough upper bound on the orders of $\sm{\one_G,\one_G}$. We can sharpen the range by refining our calculation according to individual cases of $G$ and $p$.

\subsubsection*{Case I: $G$ is $p$-regular}
Suppose $G$ is $p$-regular. Then $B_i=A_i=S^{2n_i-1}$ and $p\geq n_l$. All summands $[A_i\wedge A_j,B_k]$ in~(\ref{A_wedge_A_decomp}) are homotopy groups $\pi_{2n_i+2n_j-2}(S^{2n_k-1})$. According to Theorem~\ref{toda}, their orders are at most $p$ since
\[
2(n_i+n_j-n_k)-1\leq2(2n_l-2)-1\leq2p(p-1)-3
\]
for all $i,j$ and $k$. Therefore the order of $\sm{\imath,\imath}$ is at most $p$ and so is the order of $\sm{\one_G,\one_G}$ by Theorem~\ref{main}. McGibbon \cite{mcgibbon84} showed that $G$ is homotopy commutative if and only if either~$p>2n_l$, or $(G, p)$ is $(Sp(2), 3)$ or $(G_2, 5)$. Therefore we have the following statement.

\begin{thm}\label{regular_bound}
Let $G$ be a $p$-regular Lie group of type $(2n_1-1,\cdots,2n_l-1)$. Then the order of $\sm{\one_G,\one_G}$ is $p$ if $n_l\leq p<2n_l$, and is 1 if $p>2n_l$.
\end{thm}

\subsubsection*{Case II: $G$ is a quasi-$p$-regular $SU(n)$ and $p\geq7$}
Suppose $G=SU(n)$ is quasi-$p$-regular and $p\geq7$. Then $n_i=i+1$ and $p>\frac{n}{2}$. Let~\mbox{$t=n-p+1$} and $2\leq t\leq p$. Localized at $p$, there are homotopy equivalences
\[
SU(n)\simeq B(3,2p+1)\times\cdots\times B(2t-1,2n-1)\times S^{2t+1}\times\cdots\times S^{2p-1}
\]
and
\[
A\simeq C_3\vee\cdots\vee C_{2t-1}\vee S^{2t+1}\vee\cdots\vee S^{2p-1}
\]
For $1\leq j\leq t$ and $t+1\leq i\leq p$, let $\epsilon_i$ and $\lambda_i$ be the compositions
\[
\begin{array}{c c c}
\epsilon_i:S^{2i-1}\hookrightarrow A\overset{\imath}{\to}G
&\text{and}
&\lambda_j:C_{2j-1}\hookrightarrow A\overset{\imath}{\to}G.
\end{array}
\]
Kishimoto calculated some of their Samelson products in \cite{kishimoto09}.

\begin{thm}[Kishimoto \cite{kishimoto09}]\label{kishimoto}
Let $G$ be a quasi-$p$-regular $SU(n)$. For $2\leq j,j'\leq t$ and~\mbox{$t+1\leq i,i'\leq p$},
\begin{enumerate}
\item	the order of $\sm{\epsilon_i,\epsilon_{i'}}$ is at most $p$;
\item	if $i\neq p$ and $j\neq t$, then the order of $\sm{\epsilon_i,\lambda_j}$ is at most $p$;
\item	if $j+j'\leq p$, then $\sm{\lambda_j,\lambda_{j'}}$ is null-homotopic;
\item	if $p+1\leq j+j'\leq 2p-1$, then $\sm{\lambda_j,\lambda_{j'}}$ can be compressed into $S^{2(j+j'-p)+1}\subset SU(n)$.
\end{enumerate}
\end{thm}

Using these results we can give a bound for the order of $\sm{\one_G,\one_G}$.

\begin{thm}
For $G=SU(n)$ and $p\geq7$, let the order of $\sm{\one_G,\one_G}$ be $p^r$.
\begin{itemize}
\item	If $n>2p$, then $r=0$;
\item	If $n\leq p< 2p$, then $r=1$;
\item	If $\frac{2}{3}n+1\leq p<n$, then $1\leq r\leq2$;
\item	If $\frac{n}{2}<p\leq \frac{2}{3}n$ and $n\neq2p-1$, then $1\leq r\leq3$;
\item	If $n=2p-1$, then $1\leq r\leq6$.
\end{itemize}
\end{thm}

\begin{proof}
When $p\geq n$, $G$ is $p$-regular and we have shown the first two statements in Theorem~\ref{regular_bound}. Assume $\frac{n}{2}<p<n$. By Theorem~\ref{main}, the order of $\sm{\one_G,\one_G}$ equals the order of~$\sm{\imath,\imath}$. Since $\sm{\imath,\imath}$ is a wedge of Samelson products of $\epsilon_i$'s and $\lambda_j$'s, we need to consider the orders of $\sm{\epsilon_i,\epsilon_{i'}},\sm{\epsilon_i,\lambda_j}$ and $\sm{\lambda_j,\lambda_{j'}}$.

First, the first two statements of Theorem~\ref{kishimoto} imply that the orders of $\sm{\epsilon_i,\epsilon_{i'}}$ and $\sm{\epsilon_i,\lambda_j}$ are at most $p$ except for $\sm{\epsilon_p,\lambda_t}$. Put $n_i=p$ and $n_j=t$ in (\ref{dgm_exp_SES_SC}) to obtain the exact sequence
\[
\pi_{4p+2t-4}(B_k)\to[C_{2p+2t-2}, B_k]\to\pi_{2p+2t-2}(B_k)
\]
where $2\leq k\leq p$. According to Theorem~\ref{toda}, the two homotopy groups are trivial except for $k=t+1$ or $t=p$ and $k=2$. In the first case, $B_k$ is $S^{2t+1}$, and $\pi_{4p+2t-4}(S^{2t+1})$ and~\mbox{$\pi_{2p+2t-2}(S^{2t+1})$} are $\Z/p\Z$. In the second case, $B_k$ is $B(3,2p+1)$, and $\pi_{6p-4}(B(3,2p+1))$ and $\pi_{4p-2}(B(3,2p+1))$ are $\Z/p^2\Z$. By exactness the order of $[C_{2p+2t-2}, B_k]$ is at most~$p^2$ for~$2\leq k\leq p$ and $n\neq 2p-1$, and consequently so is the order of $\sm{\epsilon_i,\lambda_j}$.

Second, the third statement of Theorem~\ref{kishimoto} implies that $\sm{\lambda_j,\lambda_{j'}}$ is null-homotopic for $j+j'\leq p-1$. When $p\geq\frac{2}{3}n+1$, we have
\[
\begin{array}{c c c}
n\leq\frac{3}{2}(p-1)
&\text{and}
&t=n-p+1\leq\frac{1}{2}(p-1).
\end{array}
\]
In this case the order of $\sm{\lambda_j,\lambda_{j'}}$ is always 1 since $j+j'\leq2t\leq p-1$. When $\frac{n}{2}<p\leq\frac{2}{3}n$, we need to consider the orders of $\sm{\lambda_j,\lambda_{j'}}$ for $p+1\leq j+j'$. By the last statement of Theorem~\ref{kishimoto}, $\sm{\lambda_j,\lambda_{j'}}$ is in $[C_{2j-1}\wedge C_{2j'-1}, S^{2(j+j'-p)+1}]$ if $j+j'\leq2p-1$. Since $j,j'\leq t\leq p$, this can always be achieved for $n\neq2p-1$. There is a graph of short exact sequences
\[\xymatrix{
											&													&\pi_{2j+2j'+2p-4}(S^{2(j+j'-p)+1})^{\oplus2}\ar[d]\\
\pi_{2j+2j'+4p-6}(S^{2(j+j'-p)+1})\ar[r]	&[C_{2j-1}\wedge C_{2j'-1},S^{2(j+j'-p)+1}]\ar[r]	&[C',S^{2(j+j'-p)+1}]\ar[d]\\
											&													&\pi_{2j+2j'-2}(S^{2(j+j'-p)+1})
}\]
where $C'$ is the subcomplex of $C_{2j-1}\wedge C_{2j'-1}$ without the top cell. By Theorem~\ref{toda}, the three homotopy groups are $\Z/p\Z$. The exactness of the column and the row implies that the orders of $[C',S^{2(j+j'-p)+1}]$ and $[C_{2j-1}\wedge C_{2j'-1},S^{2(j+j'-p)+1}]$ are at most $p^2$ and $p^3$. Therefore~$\sm{\lambda_j,\lambda_{j'}}$ has order at most $p^3$ when $n\neq2p-1$ and $\frac{n}{2}<p\leq \frac{2}{3}n$.

We summarize the above discussion in the following table:
\begin{center}
\begin{tabular}{|>{\centering\arraybackslash}p{3.5cm}|>{\centering\arraybackslash}p{2cm}|>{\centering\arraybackslash}p{2cm}|>{\centering\arraybackslash}p{2cm}|}
\hline
								&\multicolumn{3}{ |c| }{an upper bound on the order of}\\
\cline{2-4}
								&$\sm{\epsilon_i,\epsilon_{i'}}$	&$\sm{\epsilon_i,\lambda_j}$	&$\sm{\lambda_j,\lambda_{j'}}$\\
\hline
$\frac{2}{3}n+1\leq p\leq n$	&$p$								&$p^2$							&$1$\\
\hline
$\begin{array}{c}
\frac{n}{2}<p\leq\frac{2}{3}n,\\
n\neq2p-1
\end{array}$
								&$p$								&$p^2$							&$p^3$\\
\hline
$n=2p-1$						&$p$								&$p^2$							&$p^6$\\
\hline
\end{tabular}
\end{center}
By Theorem~\ref{main}, the order of $\sm{\one_G,\one_G}$ equals the order of $\sm{\imath,\imath}$ which is the least common multiple of the orders of $\sm{\epsilon_i,\epsilon_{i'}},\sm{\epsilon_i,\lambda_j}$ and $\sm{\lambda_j,\lambda_{j'}}$, so the statement follows.
\end{proof}

\subsubsection*{Case III: $G$ is a quasi-$p$-regular exceptional Lie group and $p\geq7$}
Suppose $p\geq7$ and $G$ is a quasi-$p$-regular exceptional Lie group. That is
\begin{itemize}
\item	when $G=F_4$ or $E_6$, $p=7$ or $11$;
\item	when $G=E_7$, $p=11,13$ or $17$;
\item	when $G=E_8$, $p=11,13,17, 23$ or $29$.
\end{itemize}
For each case, we can calculate bounds on the orders of $[A_i\wedge A_j,B_k]$ for all $i,j$ and $k$ in~(\ref{A_wedge_A_decomp}) according to the CW-structure of $A$. Then we obtain the following statement.

\begin{thm}\label{thm_bound_quasi_except}
For $p\geq7$, suppose $G$ is a quasi-$p$-regular exceptional Lie group which is not~$p$-regular. Let the order of $\sm{\one_G,\one_G}$ be $p^r$. Then we have the following table
\[
\begin{array}{|c|c|c|}
\hline
G			&p				&\text{value(s) of }r\\
\hline
F_4			&7				&1\leq r\leq4\\
			&11				&1\\
\hline
E_6			&7				&1\leq r\leq4\\
			&11				&1\\
\hline
E_7			&11				&1\leq r\leq3\\
			&13				&1\text{ or }2\\
			&17				&1\\
\hline
E_8			&11				&1\leq r\leq6\\
			&13				&1\leq r\leq4\\
			&17				&1\leq r\leq3\\
			&19				&1\leq r\leq4\\
			&23,29			&1\\
\hline
\end{array}
\]
\end{thm}

\begin{remark}
It would be interesting if the precise order of $\sm{\one_G,\one_G}$ could be obtained in the case of Theorem~\ref{thm_bound_quasi_except}.
\end{remark}

\end{document}